\documentclass[10pt]{amsart}
\usepackage{amssymb}

\newtheorem{thm}{Theorem}[section]
\newtheorem{lemma}[thm]{Lemma}
\newtheorem{proposition}[thm]{Proposition}

\newtheorem{definition}[thm]{Definition}
\newtheorem{corollary}[thm]{Corollary}

\newcommand{\p}{\mathbb{P}}

\newcommand{\cf}{\mathrm{cf}}

\begin{document}

\title{Guessing models imply the singular cardinal 
	hypothesis}

\author{John Krueger}

\address{John Krueger \\ Department of Mathematics \\ 
	University of North Texas \\
	1155 Union Circle \#311430 \\
	Denton, TX 76203}
\email{jkrueger@unt.edu}

\date{March 2019}



\thanks{This material is based upon work supported by the National Science Foundation under Grant
	No. DMS-1464859.}

\begin{abstract}
	In this article we prove three main theorems: (1) guessing models are internally unbounded, 
	(2) for any regular cardinal $\kappa \ge \omega_2$, 
	$\textsf{ISP}(\kappa)$ implies that $\textsf{SCH}$ holds above $\kappa$, 
	and (3) forcing posets which have the 
	$\omega_1$-approximation property also 
	have the countable covering property. 
	These results solve open problems of Viale \cite{viale} and Hachtman and Sinapova \cite{sinapova}.
\end{abstract}

\maketitle

A major result in recent years on the consequences of forcing axioms is the theorem of M.\ Viale 
that the Proper Forcing Axiom (\textsf{PFA}) implies the Singular Cardinal Hypothesis (\textsf{SCH}). 
In fact, Viale showed that several strong 
combinatorial consequences of \textsf{PFA}, including  
the Mapping Reflection Principle (\textsf{MRP}) and 
the $P$-Ideal Dichotomy (\textsf{PID}), each imply \textsf{SCH} (\cite{vialepfa}, \cite{vialecovering}).

C.\ Weiss \cite{weiss} introduced a combinatorial principle $\textsf{ISP}(\kappa)$, for any regular 
cardinal $\kappa \ge \omega_2$, which is equivalent to $\kappa$ being supercompact in the case that $\kappa$ 
is inaccessible, but is also consistent when $\kappa$ is a small successor cardinal. 
In particular, $\textsf{ISP}(\omega_2)$ (abbreviated henceforth as \textsf{ISP}) 
is a consequence of \textsf{PFA}, and it in turn implies many of the 
strong consequences of \textsf{PFA}, such as the failure of square principles. 
Later Viale and Weiss \cite{vialeweiss} provided an 
alternative characterization of $\textsf{ISP}$ in terms of 
the existence of stationarily many elementary substructures which have a ``guessing'' property reminiscent of the 
approximation property in forcing theory.

In light of these developments, 
a natural question is whether 
$\textsf{ISP}$ implies \textsf{SCH}. 
Viale \cite{viale} made partial progress on this question by showing that \textsf{SCH} follows from an 
apparently stronger form of \textsf{ISP}, namely, the existence of stationarily 
many guessing models which are also internally unbounded. 
This result raises a number of additional 
questions, such 
as whether guessing models alone imply \textsf{SCH}, 
whether guessing models are always internally unbounded, 
and whether 
the $\omega_1$-approximation property of 
forcing posets implies the countable covering property. 
In this article we refine the results of Viale and Weiss 
described above 
and answer all of these questions in the affirmative.

\section{Guessing and covering}

For the remainder of the article, 
$N$ will usually denote an elementary substructure of $H(\theta)$ for some 
regular cardinal $\theta \ge \omega_2$, although we will not strictly require this for many of the definitions.

For a set or class $M$, a set $x \subseteq M$ is said to be \emph{bounded in $M$} if 
there exists $Y \in M$ such that $x \subseteq Y$. 

\begin{definition}
	A set $N$ is said to be \emph{guessing} 
	if for any set $x \subseteq N$ which is bounded in $N$, 
	if for all $a \in N \cap [N]^\omega$, 
	$x \cap a \in N$, 
	then there exists $E \in N$ such that $x = N \cap E$.
	\end{definition}

\begin{definition}
	For any regular cardinal $\kappa \ge \omega_2$, 
	let $\textsf{ISP}(\kappa)$ 
	be the statement that for any regular cardinal 
	$\theta \ge \kappa$, the collection of 
	guessing sets is stationary in $P_{\kappa}(H(\theta))$. 
	Let $\textsf{ISP}$ be the statement $\textsf{ISP}(\omega_2)$.
	\end{definition}

Being \emph{stationary} in $P_{\kappa}(H(\theta)) = 
\{ a \subseteq H(\theta) : |a| < \kappa \}$ means meeting every club, 
where a club is any cofinal subset of 
$P_\kappa(H(\theta))$ closed under unions of $\subseteq$-increasing 
sequences of length less than $\kappa$. 
The collection of all sets 
$N$ such that $N \cap \kappa \in \kappa$ is club in 
$P_{\kappa}(H(\theta))$, so $\textsf{ISP}(\kappa)$ implies stationarily 
many guessing models $N$ such that $N \cap \kappa \in \kappa$.

It is easy to prove from the definition that if $N$ is an elementary substructure which is guessing, 
then for any regular uncountable cardinal $\kappa \in N$, $\sup(N \cap \kappa)$ has uncountable cofinality.

\begin{definition}
	A set $N$ is said to be \emph{internally unbounded} if 
	for any countable set $x \subseteq N$ which is bounded 
	in $N$, there exists $y \in N \cap [N]^\omega$ 
	such that $x \subseteq y$.
	\end{definition}

Recall that $N$ has \emph{countable covering} if 
any countable subset of $N$ is covered by a countable 
set in $N$. 
Obviously, if $\sup(N \cap On)$ has cofinality 
$\omega$, then $N$ does not have this property, but under some typical assumptions, if $\sup(N \cap On)$ has 
uncountable cofinality then countable covering is 
equivalent to being internally unbounded.

Viale (\cite[Remark 4.3]{viale}) asked whether it is consistent to have a guessing model which is 
not internally unbounded. 
In \cite[Section 4]{JK30} 
we showed that \textsf{PFA} implies the existence of 
stationarily many 
elementary substructures $N$ of $H(\omega_2)$ 
of size $\omega_1$ such that $N$ is guessing but 
$\sup(N \cap \omega_2) = \omega$. 
Such models do not have countable covering, but they are 
internally unbounded according to Definition 1.3. 
This result solved an easy special case of Viale's 
question, but the next theorem provides the complete 
solution.

\begin{thm}
	Let $\theta \ge \omega_2$ be a regular cardinal, and suppose that  
	$N$ is an elementary substructure of $H(\theta)$ such that $\omega_1 \subseteq N$.  
	If $N$ is guessing, then $N$ is internally unbounded.
\end{thm}

\begin{proof}
	Let $x \subseteq N$ be countably infinite and bounded in $N$. 
	Fix a set $Y \in N$ such that $x \subseteq Y$. 
	Our goal is to find a countable set $y$ in $N$ such that $x \subseteq y$. 
	Observe that by elementarity, the set $[Y]^{<\omega}$ is a member of $N$. 
	Fix a bijection $g : \omega \to x$, and for each $n$ let $x_n := g[n]$. 
	Then $x_m \subseteq x_n$ for all $m < n$, $\bigcup_n x_n = x$, and 
	$\{ x_n: n < \omega \} \subseteq [Y]^{<\omega}$.
	
	We consider two possibilities. 
	The first is that there exists $\mathcal X \in N \cap [N]^\omega$ such that 
	$$
	| \mathcal X \cap \{ x_n : n < \omega \} | = \omega.
	$$
	By intersecting $\mathcal X$ with $[Y]^{<\omega}$ if necessary, we may assume without loss of 
	generality that $\mathcal X \subseteq [Y]^{<\omega}$. 
	Since $\mathcal X$ is countable and its elements are finite, $y := \bigcup \mathcal X$ is a 
	countable subset of $Y$. 
	Also, $y \in N$ by elementarity.
	
	We claim that $x \subseteq y$, which completes the proof in this case. 
	Consider $a \in x$. 
	Fix $m$ such that $a \in x_m$. 
	Since $\mathcal X \cap \{ x_n : n < \omega \}$ is infinite, we can fix $n > m$ such that 
	$x_n \in \mathcal X$. 
	Then $a \in x_m \subseteq x_n \subseteq y$, so $a \in y$.
	
	The second possibility is that for all $\mathcal X \in N \cap [N]^\omega$, 
	$\mathcal X \cap \{ x_n : n < \omega \}$ is finite. 
	Since $N$ is closed under finite subsets, for all such $\mathcal X$, 
	$\mathcal X \cap \{ x_n : n < \omega \}$ is a member of $N$. 
	In this case we will show that $x$ itself is a 
	member of $N$, which completes the proof. 
	Since $N$ is guessing, we can fix $E \in N$ such that $\{ x_n : n < \omega \} = N \cap E$. 

	Observe that $E$ is countable. 
	Otherwise there would exist an injection of $\omega_1$ into $E$ in $N$ by elementarity. 
	Since $\omega_1 \subseteq N$, it would follow that $N \cap E$ is uncountable. 
	This is impossible since $N \cap E = \{ x_n : n < \omega \}$, which is countable. 
	As $E$ is countable, $E \subseteq N$ by elementarity. 
	So $\{ x_n : n < \omega \} = N \cap E = E$. 
	Therefore, the set $\{ x_n : n < \omega \}$ is a member of $N$. 
	Thus, $x = \bigcup \{ x_n : n < \omega \}$ is a member of $N$.
\end{proof}

\begin{corollary}
	Let $\kappa \ge \omega_2$ be a regular cardinal. 
	Then $\textsf{ISP}(\kappa)$ implies that for all regular cardinals $\theta \ge \kappa$, 
	there are stationarily many $N \in P_{\kappa}(H(\theta))$ such that $N$ is 
	guessing and internally unbounded.	
	\end{corollary}

\begin{proof}
	We already know that $\textsf{ISP}(\kappa)$ implies the existence of 
	stationarily many $N \in P_{\kappa}(H(\theta))$ 
	such that $N$ is guessing and $N \cap \kappa \in \kappa$. 
	By definability, $\omega_1 \in N \cap \kappa$, and it follows that $\omega_1 \subseteq N$. 
	By Theorem 1.4, $N$ is internally unbounded.
	\end{proof}

Viale \cite[Section 7.2]{viale} proved that the existence of stationarily many internally unbounded 
guessing models implies \textsf{SCH}, but it was unknown whether guessing models alone imply \textsf{SCH}. 
This problem also appears in \cite[Section 1]{sinapova}. 
By Corollary 1.5 together with Viale's result, 
\textsf{ISP} does indeed imply \textsf{SCH}.\footnote{After announcing the results of this paper, we learned that 
S.\ Hachtman had recently and independently 
proven that \textsf{ISP} 
implies \textsf{SCH} using essentially the same 
argument as presented in this section.}

\begin{corollary}
	\textsf{ISP} implies \textsf{SCH}.
	\end{corollary}

\section{\textsf{ISP} and \textsf{SCH}}

In the previous section we showed that guessing models 
are internally unbounded, which combined with 
Viale's argument \cite[Section 7.2]{viale} proves that 
$\textsf{ISP}$ implies $\textsf{SCH}$. 
S. Hachtman and D. Sinapova \cite{sinapova} asked a more general question, which is 
whether for a regular cardinal $\kappa \ge \omega_2$, $\textsf{ISP}(\kappa)$ implies 
$\textsf{SCH}$ above $\kappa$. 
In this section we solve this problem in the affirmative. 
We note that our proof avoids the idea of internally unbounded models entirely.

We will in fact prove something a bit stronger.

\begin{thm}
	Let $\kappa \ge \omega_2$ be regular and 
	assume that $\textsf{ISP}(\kappa)$ holds. 
	Then either $\kappa$ is supercompact, or 
	$\textsf{SCH}$ holds.
	\end{thm}
	
\begin{proposition}
	Let $\kappa \ge \omega_2$ be regular and assume 
	that $\textsf{ISP}(\kappa)$ holds. 
	If $2^\omega < \kappa$, then $\kappa$ is supercompact. 
	Hence, \textsf{SCH} holds above $\kappa$.
\end{proposition}

\begin{proof}
	If $\kappa$ is strongly inaccessible and 
	$\textsf{ISP}(\kappa)$ holds, then $\kappa$ is 
	supercompact by \cite[Theorem 2.10]{weiss}. 
	And if $\kappa$ is supercompact, then 
	$\textsf{SCH}$ holds above $\kappa$ 
	by a well-known result of Solovay 
	(\cite[Theorem 20.8]{jech}).
	So it suffices to show that $\kappa$ is strongly 
	inaccessible.

	Let $\mu < \kappa$ be a cardinal 
	and we will show that $|P(\mu)| < \kappa$. 
	Using $\textsf{ISP}(\kappa)$, we can fix an elementary substructure $N$ of 
	$H(\kappa)$ of size less than $\kappa$ such that $N \cap \kappa \in \kappa$, 
	$N \cap \kappa$ is larger than $2^\omega$ and $\mu$, and $N$ is guessing. 
	It suffices to show that $P(\mu) \subseteq N$.
	
	Let $x \subseteq \mu$. 
	Then $x$ is a subset of $N$ which is bounded in $N$. 
	Consider $a \in N \cap [N]^\omega$. 
	Since $2^\omega < N \cap \kappa$, $P(a) \subseteq N$. 
	In particular, $a \cap x \in N$. 
	As $N$ is guessing, 
	it follows that there exists $E \in N$ 
	such that $x = N \cap E$. 
	By intersecting $E$ with $\mu$ if necessary, we may assume without loss of 
	generality that $E \subseteq \mu$. 
	Since $\mu$ is a subset of $N$, so is $E$, and hence 
	$x = N \cap E = E$. 
	Thus, $x \in N$, as desired.
	\end{proof}

See \cite[Theorem 2.1]{sinapova} for a similar argument.

Fix a regular cardinal $\kappa \ge \omega_2$ for the remainder of the section, and assume that 
$\textsf{ISP}(\kappa)$ holds. 
If $2^\omega < \kappa$, then $\kappa$ is supercompact, and we are done. 
Assume that $2^\omega \ge \kappa$. 
We will show that \textsf{SCH} holds. 

By a well-known theorem of Silver, the first cardinal 
for which \textsf{SCH} fails, if it exists, 
has cofinality $\omega$ (\cite[Theorem 8.13]{jech}). 
Let $\lambda$ be a singular cardinal of 
cofinality $\omega$, and assume that 
\textsf{SCH} holds below $\lambda$. 
If \textsf{SCH} fails at $\lambda$, that means that $2^\omega < \lambda$ 
and $\lambda^\omega > \lambda^+$. 
Now $2^\omega \ge \kappa$, so $\lambda > \kappa$. 
Since \textsf{SCH} holds below $\lambda$, 
an easy inductive argument shows that 
for all cardinals $\mu < \lambda$, $\mu^\omega < \lambda$ 
(\cite[Theorem 5.20]{jech}).

Putting it all together, assuming $\textsf{ISP}(\kappa)$ 
and $2^\omega \ge \kappa$, \textsf{SCH} follows from 
the statement: for all cardinals $\lambda > \kappa$ of 
cofinality $\omega$, if $\mu^\omega < \lambda$ 
for all $\mu < \lambda$, then $\lambda^\omega = \lambda^+$. 
Our proof of this statement follows along the lines of 
Viale's proof \cite[Section 7.2]{viale}, but avoids 
consideration of internal unboundedness.

\begin{lemma}[{\cite[Lemma 6]{vialecovering}}] 
	Let $\lambda > 2^\omega$ be a cardinal with cofinality $\omega$. 
	Then there exists a matrix 
	$$
	\langle K(n,\beta) : n < \omega, \ \beta < \lambda^+ \rangle
	$$
	of sets of size less than $\lambda$ satisfying:
	\begin{enumerate}
		\item for all $\beta < \lambda^+$, $\beta = \bigcup \{ K(n,\beta) : n < \omega \}$;
		\item for all $\beta < \lambda^+$ and $m < n < \omega$, 
		$K(m,\beta) \subseteq K(n,\beta)$;
		\item for all $\gamma < \beta < \lambda^+$ there exists $m < \omega$ such that 
		for all $m \le n < \omega$, $K(n,\gamma) \subseteq K(n,\beta)$;
		\item for all $x \in [\lambda^+]^\omega$ there exists $\gamma < \lambda^+$ 
		such that for all $\gamma < \beta < \lambda^+$, 
		there exists $m < \omega$ such that for all $m \le n < \omega$, 
		$K(n,\beta) \cap x = K(n,\gamma) \cap x$.
		\end{enumerate} 
	\end{lemma}

\begin{proof}
	Fix an increasing sequence of uncountable 
	cardinals $\langle \lambda_n : n < \omega \rangle$ cofinal in 
	$\lambda$. 
	By a straightforward argument, it is possible to 
	fix, for each $\beta < \lambda^+$, 
	a surjection $g_\beta : \lambda \to \beta$ satisfying 
	that for all $\gamma < \beta$ there exists $m$ 
	such that for all $n \ge m$, 
	$g_\gamma[\lambda_n] \subseteq 
	g_\beta[\lambda_n]$. 
	 
	Define $K(n,\emptyset) := \emptyset$ for all $n < \omega$. 
	Now fix $\beta < \lambda^+$ and assume that $K(n,\gamma)$ is defined for all $n < \omega$ 
	and $\gamma < \beta$. 
	Define for each $n < \omega$ 
	$$
	K(n,\beta) := g_\beta[\lambda_n] \cup \bigcup \{ K(n,\gamma) : \gamma \in g_\beta[\lambda_n] \}.
	$$
	This completes the definition. 
	It is easy to prove by induction that (1), (2), and 
	(3) hold, and each $K(n,\beta)$ has 
	size at most $\lambda_n$.

	For (4), fix $x \in [\lambda^+]^\omega$. 
	For each $\beta < \lambda^+$, define a function $f_\beta : \omega \to P(x)$ by 
	$f_\beta(n) := K(n,\beta) \cap x$. 
	Observe that there are $2^\omega$ many possibilities for such a function $f_\beta$. 
	Since $2^\omega < \lambda$, we can fix a set $S \subseteq \lambda^+$ of size $\lambda^+$ 
	and a function $f$ such that for all $\beta \in S$, $f_\beta = f$.
	Let $\gamma := \min(S)$.
	
	To verify that (4) holds for $x$, 
	consider $\beta > \gamma$. 
	Let $\xi := \min(S \setminus \beta)$. 
	Using (3), fix $m$ such that for all $n \ge m$, 
	$$
	K(n,\gamma) \subseteq K(n,\beta) \subseteq K(n,\xi).
	$$
	In particular, 
	$K(n,\gamma) \cap x \subseteq K(n,\beta) \cap x$. 
	For the reverse inclusion, 
	$$
	K(n,\beta) \cap x \subseteq K(n,\xi) \cap x = 
	f_\xi(n) = f(n) = f_{\gamma}(n) = 
	K(n,\gamma) \cap x.
	$$
\end{proof}

\begin{lemma}[{\cite[Fact 9]{vialecovering}}]
	Let $\lambda > 2^\omega$ be a singular cardinal 
	with cofinality $\omega$ 
	such that 
	for all cardinals $\mu < \lambda$, $\mu^\omega < \lambda$. 
	Fix $\langle K(n,\beta) : n < \omega, \ \beta < \lambda^+ \rangle$ as described in Lemma 2.3. 
	Assume that there exists a set $S \subseteq \lambda^+$ of size $\lambda^+$ 
	such that for all $x \in [S]^\omega$, there exists $n < \omega$ and $\beta < \lambda^+$ 
	such that $x \subseteq K(n,\beta)$. 
	Then $\lambda^\omega = \lambda^+$.
	\end{lemma}

\begin{proof}
	Since $S$ has size $\lambda^+$, the cardinality of $[S]^\omega$ is equal to $(\lambda^+)^\omega$, 
	which in turn equals $\lambda^\omega$. 
	So it suffices to show that $[S]^\omega$ has cardinality $\lambda^+$. 
	By assumption, every member of $[S]^\omega$ is a subset of $K(n,\beta)$ for some $n < \omega$ 
	and $\beta < \kappa^+$. 
	Thus,
	$$
	[S]^\omega \subseteq \bigcup \{ [K(n,\beta)]^\omega : n < \omega, \ \beta < \lambda^+ \}.
	$$
	Now each $K(n,\beta)$ has cardinality less than $\lambda$, so by our assumptions, 
	$[K(n,\beta)]^\omega$ has cardinality less than $\lambda$. 
	Thus, the union in the above inclusion has cardinality $\lambda^+$.
	\end{proof}

Assume that $\textsf{ISP}(\kappa)$ holds, and let $\lambda > \kappa$ be a singular cardinal 
of cofinality $\omega$ such that for all $\mu < \lambda$, $\mu^\omega < \lambda$. 
We will prove that $\lambda^\omega = \lambda^+$.

	Fix $\mathcal K = \langle K(n,\beta) : n < \omega, \ \beta < \lambda^+ \rangle$ as described in Lemma 2.3. 
	In order to show that $\lambda^\omega = \lambda^+$, by Lemma 2.4 it suffices to show that 
	there exists a set $S \subseteq \lambda^+$ of size $\lambda^+$ 
	such that for all $x \in [S]^\omega$, 
	there exists $n < \omega$ and $\beta < \lambda^+$ 
	such that $x \subseteq K(n,\beta)$.

	Using $\textsf{ISP}(\kappa)$, 
	fix an elementary substructure $N$ of $H(\lambda^{++})$ 
	of size less than $\kappa$ such that $N \cap \kappa \in \kappa$, 
	$\mathcal K \in N$, and $N$ is guessing. 
	For each $x \in [\lambda^+]^\omega$, let $\gamma_x < \lambda^+$ be the minimal ordinal 
	satisfying that for all 
	$\gamma_x < \beta < \lambda^+$, there exists $n$ such that for all $m \ge n$, 
	$K(m,\beta) \cap x = K(m,\gamma_x) \cap x$. 
	Observe that $\langle \gamma_x : x \in [\lambda^+]^\omega \rangle$ is a member of 
	$N$ by elementarity.
	
	Consider $x \in N \cap [\lambda^+]^\omega$. 
	Then $\gamma_x \in N \cap \lambda^+$. 
	So there exists $n$ such that for all $m \ge n$, 
	$$
	K(m,\sup(N \cap \lambda^+)) \cap x = K(m,\gamma_x) \cap x.
	$$
	Since $x$, $\gamma_x$, and $\mathcal K$ are in $N$, 
	$K(m,\gamma_x) \cap x$ is a member of $N$. 
	Therefore, 
	$$
	K(m,\sup(N \cap \lambda^+)) \cap x \in N.
	$$
	Now for each $x \in N \cap [\lambda^+]^\omega$, fix the smallest integer 
	$k_x$ satisfying that for all $m \ge k_x$, 
	$K(m,\sup(N \cap \lambda^+)) \cap x$ is in $N$.
	
	We claim that if $x$ and $y$ are in $N \cap [\lambda^+]^\omega$ 
	and $x \subseteq y$, then $k_x \le k_y$. 
	By the minimality of $k_x$, it suffices to show that 
	for all $m \ge k_y$, $K(m,\sup(N \cap \lambda^+)) \cap x \in N$. 
	Let $m \ge k_y$. 
	Then $K(m,\sup(N \cap \lambda^+)) \cap y \in N$. 
	Since $x$ is in $N$ and $x \subseteq y$, we have that 
	$K(m,\sup(N \cap \lambda^+)) \cap x = (K(m,\sup(N \cap \lambda^+)) \cap y) \cap x$ 
	is in $N$.
	
	Next, we claim that the collection of integers 
	$$
	A := \{ k_x : x \in N \cap [\lambda^+]^\omega \}
	$$
	is finite. 
	Suppose for a contradiction that $A$ is infinite. 
	For each $n \in A$, fix $x_n \in N \cap [\lambda^+]^\omega$ such that $n = k_{x_n}$. 
	Now define, for each $n \in A$, $y_n := \bigcup \{ x_k : k \in A \cap (n+1) \}$, 
	which is in $N \cap [\lambda^+]^\omega$. 
	Observe that if $m < n$ are in $A$, then $y_m \subseteq y_n$. 
	Also, for each $n \in A$, $x_n \subseteq y_n$, and therefore by the previous paragraph, 
	$n = k_{x_n} \le k_{y_n}$. 
	By thinning out the sequence $\langle y_n : n \in A \rangle$ if necessary, it is easy to find a 
	sequence $\langle z_n : n < \omega \rangle$ of distinct sets in $N \cap [\lambda^+]^\omega$ satisfying that 
	for all $m < n$, $z_m \subseteq z_n$ and 
	$k_{z_m} < k_{z_n}$.
	
	We now consider two possibilities, both of which will lead to a contradiction. 
	First, assume that there exists a countable set $\mathcal X \in N$ such that 
	$$
	|\mathcal X \cap \{ z_n : n < \omega \}| = \omega.
	$$
	By intersecting $\mathcal X$ with $[\lambda^+]^\omega$ if necessary, we may assume without loss of generality 
	that $\mathcal X \subseteq [\lambda^+]^\omega$. 
	Since $\mathcal X$ is countable and consists of countable sets, $x^* := \bigcup \mathcal X$ is in 
	$N \cap [\lambda^+]^\omega$. 
	We claim that for all $m < \omega$, $z_m \subseteq x^*$. 
	Indeed, given $m$, we can find $n \ge m$ such that $z_n \in \mathcal X$. 
	Then $z_m \subseteq z_n \subseteq \bigcup \mathcal X = x^*$. 
	Now for all $n < \omega$, $z_n \subseteq x^*$ implies that $k_{z_n} \le k_{x^*}$. 
	This is impossible, since $\{ k_{z_n} : n < \omega \}$ is unbounded in $\omega$, 
	whereas $k_{x^*} < \omega$.

	Secondly, assume that for all countable sets $\mathcal X \in N$, 
	$\mathcal X \cap \{ z_n : n < \omega \}$ is finite. 
	Then in particular, for all countable sets $\mathcal X \in N$, 
	$\mathcal X \cap \{ z_n : n < \omega \}$ is a member of $N$. 
	Also note that this assumption implies that 
	$\{ z_n : n < \omega \}$ is not in $N$, 
	for otherwise we could let $\mathcal X$ 
	be equal to it and get a contradiction. 
	Since $N$ is guessing, it follows that there exists $E \in N$ such that 
	$\{ z_n : n < \omega \} = N \cap E$. 
	In particular, $N \cap E$ is countable. 
	Since $\omega_1 \subseteq N$, this implies that $E$ is countable, for otherwise by elementarity 
	$N \cap E$ would be uncountable. 
	Therefore, $E \subseteq N$. 
	So $\{ z_n : n < \omega \} = N \cap E = E$, and hence $\{ z_n : n < \omega \}$ is a member of $N$, 
	which is a contradiction.

	This concludes the proof that the set 
	$A = \{ k_x : x \in N \cap [\lambda^+]^\omega \}$ is finite. 
	Let $n^*$ be the largest member of $A$. 
	Then for all $x \in N \cap [\lambda^+]^\omega$, $k_x \le n^*$ implies that for all $m \ge n^*$, 
	$K(m,\sup(N \cap \lambda^+)) \cap x \in N$. 
	It easily follows that for all $m \ge n^*$, for any countable set $Y \in N$, 
	$K(m,\sup(N \cap \lambda^+)) \cap Y \in N$. 
	Since $N$ is guessing, for all $m \ge n^*$ there exists a set $E_m \in N$ 
	such that $N \cap K(m,\sup(N \cap \lambda^+)) = N \cap E_m$. 
	By intersecting $E_m$ with $\lambda^+$ if necessary, we may assume without loss 
	of generality that $E_m \subseteq \lambda^+$.

	Since $\sup(N \cap \lambda^+)$ is equal to 
	$\bigcup \{ K(m,\sup(N \cap \lambda^+)) : m < \omega \}$, we have that 
	$$
	N \cap \sup(N \cap \lambda^+) = \bigcup \{ N \cap K(m,\sup(N \cap \lambda^+)) : m < \omega \}.
	$$
	As $\cf(\sup(N \cap \lambda^+))$ is uncountable, there exists $m \ge n^*$ such that 
	$N \cap K(m,\sup(N \cap \lambda^+)) = N \cap E_m$ is unbounded in $\sup(N \cap \lambda^+)$. 
	By elementarity, it easily follows that the set $S := E_m$ is unbounded in $\lambda^+$. 
	To complete the proof, it suffices to show that for all 
	$x \in [S]^\omega$, there exists $n < \omega$ and $\beta < \lambda^+$ 
	such that $x \subseteq K(n,\beta)$. 
	Since $S \in N$, by elementarity it suffices to show that for all $x \in N \cap [S]^\omega$, 
	there exists $n < \omega$ and $\beta < \lambda^+$ such that $x \subseteq K(n,\beta)$. 
	Fix $x \in N \cap [S]^\omega$. 
	Then $x \subseteq N \cap S = N \cap E_m = N \cap K(m,\sup(N \cap \lambda^+))$. 
	By elementarity, there exists $\beta \in N \cap \lambda^+$ such that $x \subseteq K(m,\beta)$.

\section{Approximation and covering}

In Section 1 we saw that guessing implies 
internally unbounded for elementary substructures. 
In this section we provide analogous results concerning the 
approximation property implying the covering property, for models and forcing posets.

\begin{definition}
	Let $\kappa$ be a regular uncountable cardinal. 
	Let $W_1 \subseteq W_2$ be transitive (sets or classes) with $\kappa \in W_1$. 
	\begin{enumerate}
		\item  The pair $(W_1,W_2)$ is said to have the \emph{$\kappa$-approximation property} 
	provided that whenever $X \in W_2$ is a bounded subset of $W_1$, 
	if $X \cap y \in W_1$ for any set $y \in W_1$ such that $W_1 \models |y| < \kappa$, 
	then $X \in W_1$;
		\item The pair $(W_1,W_2)$ is said to have the \emph{$\kappa$-covering property} if whenever 
		$X \in W_2$ is a bounded subset of $W_1$, 
		if $W_2 \models |X| < \kappa$, then there exists $Y \in W_1$ such that 
		$W_1 \models |Y| < \kappa$ and $X \subseteq Y$.
	\end{enumerate}
\end{definition}

\begin{definition}
	Let $\kappa$ be a regular uncountable cardinal and $\p$ a forcing poset. 
	We say that $\p$ has the \emph{$\kappa$-approximation property} if $\p$ forces that 
	$(V,V^\p)$ has the $\kappa$-approximation property, and has the \emph{$\kappa$-covering property} 
	if $\p$ forces that $(V,V^\p)$ has the $\kappa$-covering property. 
\end{definition}

\begin{thm}
	Let $\kappa$ be a regular uncountable cardinal and $W_1 \subseteq W_2$ be 
	transitive models of \textsf{ZFC} minus power set such that $\kappa \in W_1$. 
	Assume that for all $W_2$-cardinals $\mu < \kappa$, any subset of $W_1$ which is a member of 
	$W_2$ and has $W_2$-cardinality less than $\mu$ is a member of $W_1$. 
	If $(W_1,W_2)$ has the $\kappa$-approximation property, then it has the $\kappa$-covering property.
	\end{thm}

\begin{proof}
	Let $x \in W_2$ satisfy that $W_2 \models |x| < \kappa$ and $x \subseteq Y$ 
	for some $Y \in W_1$. 
	We will prove that $x$ is covered by some set in $W_1$ which has $W_1$-cardinality less than $\kappa$. 
	Define $\mu := |x|^{W_2}$. 
	Since $x$ has cardinality $\mu$ in $W_2$, fix a bijection $g : \mu \to x$ in $W_2$, and define for 
	each $i < \mu$ $x_i := g[i]$. 
	Then the sequence $\langle x_i : i < \mu \rangle$ is in $W_2$, is $\subseteq$-increasing, and has 
	union equal to $x$. 
	Moreover, each $x_i$ has size less than $\mu$ in $W_2$, hence is in $W_1$ by our assumptions, 
	and has $W_1$-cardinality less than $\mu$.

	We consider two possibilities. 
	First, assume that there exists a set $\mathcal X \in W_1$ of $W_1$-cardinality less than $\kappa$ such that 
	$$
	W_2 \models | \mathcal X \cap \{ x_i : i < \mu \} | = \mu.
	$$
	By intersecting $\mathcal X$ with $([Y]^{<\mu})^{W_1}$ if necessary, we may assume without loss of generality that 
	$\mathcal X \subseteq ([Y]^{<\mu})^{W_1}$. 
	Since $\mu < \kappa$, $z := \bigcup \mathcal X$ is a subset of $Y$ of $W_1$-cardinality less than $\kappa$. 
	For all $i < \mu$, there exists $j > i$ in $\mu$ such that $x_j \in \mathcal X$, 
	so $x_i \subseteq x_j \subseteq z$. 
	Hence, $z$ is a member of $W_1$ of $W_1$-cardinality less than $\kappa$ such that 
	$x = \bigcup \{ x_i : i < \mu \}$ is a subset of $z$, as required.
	
	Secondly, assume that for all $\mathcal X \in W_1$ of $W_1$-cardinality less than $\kappa$, 
	$$
	W_2 \models | \mathcal X \cap \{ x_i : i < \mu \} | < \mu.
	$$
	Since each $x_i$ is a member of $W_1$, it follows from our assumptions that 
	$\mathcal X \cap \{ x_i : i < \mu \}$ is a member of $W_1$. 
	Also, the set $\{ x_i : i < \mu \}$ is a subset of a member of $W_1$, namely 
	the set $([Y]^{<\mu})^{W_1}$. 
	As the pair $(W_1,W_2)$ has the $\kappa$-approximation property, it follows that 
	$\{ x_i : i < \mu \}$ is a member of $W_1$. 
	This is impossible, since letting $\mathcal X$ be equal 
	to $\{ x_i : i < \mu \}$, 
	we get a contradiction 
	to the assumption of this case.
\end{proof}

\begin{corollary}
	Let $\lambda$ be a regular cardinal and $\p$ 
	a forcing poset. 
	Assume that $\p$ is $<\!\lambda$-distributive. 
	If $\p$ has the $\lambda^+$-approximation property, 
	then $\p$ has the $\lambda^+$-covering property. 	
	\end{corollary}

\begin{proof}
	By Theorem 3.3, it suffices to show that $\p$ 
	preserves $\lambda^+$. 
	If not, then there exists a cofinal set 
	$x \subseteq (\lambda^+)^V$ in $V^{\p}$ of order 
	type at most $\lambda$. 
	If $a \in V$ has $V$-cardinality less than 
	$(\lambda^+)^V$, 
	then $a \cap x$ is bounded in $(\lambda^+)^V$, 
	and hence has order type less than $\lambda$. 
	As $\p$ is $<\!\lambda$-distributive, 
	$a \cap x \in V$. 
	Since $\p$ has the $\lambda^+$-approximation property, 
	it follows that $x \in V$, which is impossible.
	\end{proof}

Observe that if $\kappa$ is weakly inaccessible or 
the successor of a singular cardinal, then a forcing poset 
$\p$ being $<\mu$-distributive for all cardinals 
$\mu < \kappa$ implies that $\p$ is $<\!\kappa$-distributive, 
and hence has the $\kappa$-covering property. 
That is why we restricted the statement of the corollary 
to successors of regulars.

\begin{corollary}
	If $\p$ is a forcing poset which has the $\omega_1$-approximation property, then $\p$ has the 
	$\omega_1$-covering property. 
	\end{corollary}

This follows from the fact that $\p$ forces 
that $V^{<\omega} \cap V^\p \subseteq V$.

\bibliographystyle{plain}
\bibliography{paper36}

\end{document}